\documentclass[12pt]{article}
	\usepackage{amsmath,fullpage,amsthm,amssymb,xcolor,graphicx}
\usepackage{amsthm,fullpage}
\usepackage{amsmath}
\usepackage{amssymb}
\usepackage{mathtools}
\usepackage{float}
\usepackage{algorithm}
\usepackage{graphicx}

\theoremstyle{plain}
\numberwithin{equation}{section}

\newtheorem{theorem}{Theorem}[section]
\newtheorem{corollary}[theorem]{Corollary}
\newtheorem{lemma}[theorem]{Lemma}

\newtheorem{conj}{Conjecture}[section]
\newtheorem{remark}{Remark}[section]

\DeclareMathOperator{\spec}{spec}
\DeclareMathOperator{\sgn}{sgn}

\DeclarePairedDelimiter\floor{\lfloor}{\rfloor}

\title{ Signed spectral Tura\'{n} type theorems}

\author{M. Rajesh Kannan\thanks{Department of Mathematics, Indian Institute of Technology Hyderabad, Hyderabad 502284, India. Email: rajeshkannan@math.iith.ac.in, rajeshkannan1.m@gmail.com }\and Shivaramakrishna Pragada\thanks{Department of Mathematics, Simon Fraser University, Burnaby, BC, V5A 1S6, Canada. Email:  shivaramkratos@gmail.com }
}

\begin{document}

    \maketitle
    % {
    %   \hypersetup{linkcolor=black}
    %   \tableofcontents
    % }

    % ---------------------------------
    % Table of contents for comments
    % \listoftodos[\textsf{Comments}]
    % ---------------------------------

\begin{abstract}
A signed graph $\Sigma = (G, \sigma)$ is a graph where the function $\sigma$ assigns either $1$ or $-1$ to each edge of the simple graph $G$.  The adjacency matrix of $\Sigma$, denoted by $A(\Sigma)$, is defined canonically. In a recent paper, Wang et al. extended the spectral bounds of Hoffman and Cvetkovi\'{c}  for the chromatic number of signed graphs. They proposed an open problem related to the balanced clique number and the largest eigenvalue of a signed graph. We solve a strengthened version of this open problem. As a  byproduct, we give alternate proofs for some of the known classical bounds for the least eigenvalues of unsigned graphs. We extend the Tur\'{a}n's inequality for the signed graphs. Besides, we study the  Bollob\'{a}s and Nikiforov conjecture for the signed graphs and show that the conjecture need not be true for the signed graphs. Nevertheless, the conjecture holds for signed graphs under some assumptions. Finally, we study some of the relationships between the number of signed walks and the largest eigenvalue of a signed graph.
\end{abstract}


% --------------------------------------
\section{Introduction}

A signed graph $\Sigma$ is a pair $(G, \sigma)$, where $G =(V, E)$ is an undirected graph, called the underlying unsigned graph, and $\sigma:E\rightarrow\{-1, +1\}$ is the sign function. Let $e_{ij}$ denote the edge between the vertices $v_i$ and $v_j$ of $G$. The \emph{adjacency matrix} of a signed graph $ \Sigma=(G, \sigma)$ is  a symmetric  matrix, denoted by $ A(\Sigma)$ and its $ (i,j)$-th  entry is defined as follows:
$$a_{ij}=\begin{cases}
    \sigma(e_{ij})&\text{if} \mbox{ $v_i\sim v_j$},\\
    0&\text{otherwise.}\end{cases}$$

    Let $\lambda_1(\Sigma) \geq \lambda_2(\Sigma) \geq \dots \geq \lambda_n(\Sigma) $ be the eigenvalues of $A(\Sigma)$. The spectrum of $A(\Sigma)$ is the multiset of eigenvalues of $A(\Sigma)$.
    The spectrum and the spectral radius of $\Sigma$ are the spectrum and the spectral radius of $ A(\Sigma) $  and denoted by $\spec(\Sigma)$ and $\rho(\Sigma)$, respectively. Let $(G, 1)$ and $(G, -1)$ denote the signed graphs with all the edge signs are equal to $1$ and $-1$, respectively.  For more details about the notion of signed graphs, we refer to \cite{harary1955, zas1982,zassurvey}.

%   The degree of the vertex $v_j$ is denoted by $d_j$. By slight abuse of notation, we write $\sigma(e_{ij})$ as only $\sigma_{ij}$, to represent the gain on edge $e_{ij}$. We define a diagonal matrix $\mathbf{D}(G) = \diag(d_1, d_2, ..., d_n )$, where $d_i$ is the degree of vertex $v_i$ in the underlying graph $G$. The  Laplacian matrix $\mathbf{L}(\Sigma)$ is defined as $\mathbf{L}(\Sigma) = \mathbf{D}(G)-\mathbf{A}(\Sigma)$, where $\mathbf{D}(G) = \diag(d_1, d_2, ..., d_n )$ is a diagonal matrix and $d_i$ is the degree of vertex $v_i$ in the underlying graph $G$. It is clear from the above definition that $\mathbf{L}(\Sigma)$ is symmetric and positive semi-definite.

The \textit{sign of a cycle }(with some orientation) $C = v_1 v_2 \ldots v_l v_1$, denoted by $\sigma(C)$, is defined as the product of the signs of its edges, that is $$\sigma(C) = \sigma(e_{12}) \sigma(e_{23}) \cdots \sigma(e_{(l-1)l}) \sigma(e_{l1}),$$ where $e_{ij}$ is the edge between the vertices $v_i$ and $v_j$.
A cycle $C$ is  \textit{neutral} if $\sigma(C) = 1$, and a signed graph is  \textit{balanced} if all its cycles are neutral. A function from the vertex set of $G$ to the set $\{1 ,-1\}$ is called a \textit{switching function}. Two signed graphs $\Sigma_1 = (G, \sigma_1)$ and $\Sigma_2 = (G, \sigma_2)$ are \textit{switching equivalent}, written as $\Sigma_1 \sim \Sigma_2$, if there is a switching function $\eta: V \to \{1 ,-1\}$ such that $$\sigma_2(e_{ij})=\eta(v_i)^{-1}\sigma_1(e_{ij})\eta(v_j).$$
The switching equivalence of two signed graphs can be defined in the following equivalent way: Two signed graphs $\Sigma_1 = (G, \sigma_1)$ and $\Sigma_2 = (G, \sigma_2)$ are \textit{switching equivalent}, if there exists
a diagonal matrix $D_{\eta}$ with diagonal entries from $\{1, -1\}$ such that
\begin{align}\label{eq: switching equi}
    A(\Sigma_2) = D_{\eta}^{-1}A(\Sigma_1)D_{\eta}.
\end{align}
Switching equivalence preserves connectivity and balance.

%The least Laplacian eigenvalue has a special role in the spectral theory of signed graphs. In fact, if the least eigenvalue is zero, then $\Sigma = (G, \sigma)$ is switching equivalent to $(G, 1)$, and $\mathbf{L}(\Sigma)$ is similar to $\mathbf{L}(G)$.  Similarly, if $\Sigma$ is switching equivalent to $(G, -)$, then $\mathbf{L}(\Sigma)$ is similar to $\mathbf{Q}(G) =\mathbf{D}(G) +\mathbf{A}(G)$, the signless Laplacian of $G$, and then we have the signless Laplacian theory of (usual) graphs \cite{kannan2020normalized, reff2012spectral}.

A complete subgraph of an unsigned graph $G$ is a called a \textit{clique} in $G$. The \textit{clique number} of $G$, denoted by $\omega(G)$, is the number of vertices in a largest clique in $G$.  Let $S_{n}^{+} =  \Big\{(x_1, \dots, x_n) \in \mathbb{R}^n : \sum x_i = 1,   x_i \geq 0 ~\mbox{for }~ i = 1 ,\dots, n \Big\}$ and $S_n^{\pm} = \Big\{ (x_1, \dots, x_n)\in \mathbb{R}^n: \sum_{i=1}^{n} \vert x_i \vert =1\Big\}.$

In the seminal paper \cite{motzkin1965maxima}, Motzkin and Straus proved the following  result.
\begin{theorem} [{\cite[Theorem 1]{motzkin1965maxima}}]\label{MS theorem}
    Let  $G$ be a graph on $n$ vertices with the clique number  $\omega(G)$. If $A(G)$ is the adjacency matrix $G$, then$$ \max_{x\in S_n^{+}} x^TA(G)x = \frac{\omega(G)-1}{\omega(G)}. $$
\end{theorem}

Using this result, Nikiforov came up with a proof of the following theorem, which was conjectured by Edwards and Elphick \cite{edwars-elphick-1983}.

\begin{theorem}[{\cite[Theorem 2.1]{nikiforov2002some}}]\label{nikiforov_2m}
    Let  $G$ be a graph with $n$ vertices and $m$ edges.  Let $\lambda_1(G)$  be the largest  eigenvalue of  $A(G)$, and $\omega(G)$ be the clique number  of $G$.  Then $$ \lambda^2_1(G) \leq 2m\frac{\omega(G)-1}{\omega(G)}. $$
\end{theorem}

The balanced clique number of a signed graph $\Sigma$, denoted by $\omega_b(\Sigma)$, is the maximum number of vertices in a balanced complete subgraph. The MS-index of a signed graph $\Sigma = (G, \sigma)$, denoted by $\mu(\Sigma)$, is defined as follows:
$$ \mu(\Sigma) = \max_{x \in S_{n}^{\pm}} \sum_{i \sim j} \sigma(e_{ij}) x_ix_j = \max_{x \in S_{n}^{\pm}} \frac{x^TA(\Sigma)x}{2}. $$ In \cite{wang2021eigenvalues}, Wang et al. extended the Motzkin-Straus theorem for the signed graphs .

\begin{theorem}[{\cite[Theorem 5]{wang2021eigenvalues}}]\label{Signed-MS theorem}
Let $\Sigma = (G, \sigma)$ be a signed graph with balanced clique number $\omega_{b}(\Sigma)$.
    Then $$ \mu(\Sigma) = \frac{1}{2} \bigg(\frac{\omega_b(\Sigma)-1}{\omega_b(\Sigma)}\bigg). $$
\end{theorem}

As a consequence of the above theorem, Wang et al. extended Cvetkovi\'{c}'s theorem for the signed graphs.

\begin{theorem}[{\cite[Proposition 5]{wang2021eigenvalues}}]\label{Balance-clique-n}
    Let $\Sigma = (G, \sigma)$ be a signed graph with $n$ vertices, and balanced clique number $\omega_{b}(\Sigma)$ . Let $\lambda_1(\Sigma) $ be the largest eigenvalue of $A(\Sigma)$. Then $$ \lambda_1(\Sigma) \leq n\bigg(1-\frac{1}{\omega_b(\Sigma)}\bigg). $$
\end{theorem}
In the same paper, Wang et al. explained the difficulties in extending Nikiforov's theorem (Theorem \ref{nikiforov_2m}) for the signed graphs and mentioned it as an open problem. One of the main objectives of this article is to give a proof of a strengthened version of this open problem for the signed graph (see Theorem \ref{nikiforov-signed-general}). We give an alternate proof for Theorem \ref{Balance-clique-n} too. As a byproduct, we derive a spectral lower bound for the edge bipartiteness of a simple graph and deduce the upper bounds for the least eigenvalues of unsigned graphs obtained by Constantine \cite{gregory1985laa}. The \textit{edge bipartiteness} of an unsigned graph $G$ is the least number of edges whose   deletion yields a bipartite graph \cite{shaun2012laa}. These results are given in Section \ref{sec:frustration-bounds}.

Let $K_n$ denote the complete graph on $n$ vertices. In \cite{bollobas2007cliques}, Bollob\'{a}s and Nikiforov proposed the following conjecture. If $G$ is a $K_{r+1}$-free graph on at least $r+1$ vertices with $m$ edges, and  $\lambda_1(G) \geq \lambda_2(G) \geq \dots \geq \lambda_n(G)$ are the eigenvalues of $A(G)$, then $\lambda_1(G)^2 + \lambda_2(G)^2 \leq 2m\Big( \frac{r-1}{r}\Big).$  Recently Lin et al. \cite{lineigenvalues} confirmed this conjecture for the triangle free graphs.  We show that the signed graph version of this conjecture need not be true for the signed graphs even for the triangle free graphs. We prove the signed analogue of this conjecture for the triangle free graphs whose largest eigenvalue is the spectral radius. This is done in Section \ref{signed-niki}.

Given an unsigned graph $G$, a $k$-walk is a sequence of vertices $v_1,\dots,v_k$ of $G$ such that $v_i$ is adjacent to $v_{i+1}$ for all $i = 1,\dots,k-1$. The length of a walk is the number of edges it has, counting repeated edges as many times as they appear. So the length of a $k$-walk is $k-1$. Let  $w_r(G)$ be the number of $r$-walks in $ G$. Note that, if $G$ is an unsigned graph, then $w_r(G) = e^TA^{r-1}(G)e$, where $e= (1,1,\dots,1)^T$.   Nikiforov proved the following result connecting the number of walks in $G$, chromatic number of $G$ and the spectral radius of the adjacency matrix $A(G)$ of $G$.

\begin{theorem}[{\cite[Theorem 5 $\&$ Theorem 14]{nikiforov2006walks}}]\label{nikiforov-walks}
    Let $G$ be a graph with $n$ vertices and $m$ edges. Let $A(G)$ be its adjacency matrix and $\omega(G)$ be its clique number.  Let $\lambda_1(G) \geq \lambda_2(G) \geq \dots \geq \lambda_n(G)$ be the eigenvalues of the adjacency matrix of $G$.
    Then for $r>0$ and odd $q>0$, $$ \frac{w_{q+r}(G)}{w_q(G)} \leq \lambda^r_1(G) \leq w_r(G)\frac{\omega(G)-1}{\omega(G)}. $$
\end{theorem}
In Section \ref{signedwalk-largesteigen}, we extend the above result for the signed graphs. As a corollary to this result, we obtain Theorem \ref{Balance-clique-n}.

%   ================================================

\section{Preliminaries}

%   In this section, we introduce some of the needed notations and results.
% Nikiforov \cite{nikiforov2020turan} defined the  MS index of a graph  $G$ as follows:
%   $$ \mu(G) = \max_{\mathbf{x} \in S_{n}^{+}} \sum_{i \sim j} x_ix_j = \max_{\mathbf{x} \in S_{n}^{+}} \frac{x^TA(G)x}{2} $$
%   Using this notion, the conclusion of  Theorem \ref{MS theorem} can be stated as $ \mu(G) = \frac{1}{2}\bigg(1 - \frac{1}{\omega(G)}\bigg). $

In this section, we collect the needed matrix theoretic preliminaries about doubly stochastic matrices. For more details on related knowledge, we refer the reader to  \cite{bhatia1997book, zhan2013matrix}. A non-negative square matrix is \textit{doubly stochastic} if the sum of the entries in every row and every column is $1$, and it is \textit{doubly substochastic} if the sum
of the entries in every row and every column is at most $1$.

%A square matrix is called \textit{a weak-permutation matrix} if every row and every column has at most one non-zero entry and all the non-zero entries (if any) are $1$.

%\begin{lemma}[{\cite[ Theorem 3.11]{zhan2013matrix}}]\label{Birkhoff-Neumann-thm}
%   Every doubly substochastic matrix is a convex combination of
%   weak-permutation matrices.
%\end{lemma}
%
%For $\textbf{x} \in \mathbb{R}^n$ and $1\leq p < \infty$,  the \textit{$p$-norm} of $\textbf{x}$ is defined as $\Vert x \Vert_p = \big\{\sum_{i=1}^{n}\vert x_i\vert^p\big\}^\frac{1}{p}.$
%
%\begin{lemma}$(\textbf{Minkowski inequality})$.
%   Let $\mathbf{x},\mathbf{y}\in \mathbb{R}^n_{+}$. If $p>1$, then $\Vert\mathbf{x}+\mathbf{y}\Vert_p \leq \Vert\mathbf{x}\Vert_p + \Vert\mathbf{y}\Vert_p$. Moreover, if both $\mathbf{x}$ and $\mathbf{y}$ are non-zero vectors, equality holds if and only if there exists $\alpha \geq 0$ such that $\mathbf{x} = \alpha \mathbf{y}$.
%\end{lemma}
%
%\begin{lemma}\label{Multiple-minkowski}$(\textbf{multiple Minkowski inequality})$.
%   Let $\{\mathbf{x}^{1}, \dots , \mathbf{x}^{k} \}\subseteq \mathbb{R}^n_{+}$ with $k \geq 2$. If $p>1,$ then $$ \Vert\sum_{i=1}^{k} \mathbf{x}^{i}\Vert_p \leq \sum_{i=1}^{k} \Vert\mathbf{x}^{i}\Vert_p.$$
%   Moreover, if $\mathbf{x}^i \neq 0$ for all $i$, then equality holds if and only if there exists $\alpha_{ij} > 0$ such that $\mathbf{x}^i =\alpha_{ij}\mathbf{x}^j$ for all $i,j$ with $i \neq j$.
%\end{lemma}
For a vector $x \in \mathbb{R}^n$, let $x^\downarrow$ denote the vector obtained by rearranging the coordinates of $x$ in the decreasing order. Given two vectors $x, y \in \mathbb{R}^n $,  we say that  $x$ is \textit{weakly majorised} by $y$, denote by $x \prec_w y$ if
\begin{gather*}
    \sum_{j=1}^{k} {x_j}^\downarrow \leq \sum_{j=1}^{k} {y_j}^\downarrow ~\mbox{ for all}~ 1 \leq k \leq n.
\end{gather*}
If $x \prec_w y \ \text{and}  \ \sum_{i=1}^{n} x_i^\downarrow = \sum_{i=1}^{n} y_i^\downarrow,$
then we say that $x$ is \textit{majorized} by $y$ and denote it by $x \prec y$.

%We also use the definition of ‘a vector is weakly majorized by the other one’ as follows, where we rearrange the components of $\mathbf{x} = (x_1,x_2,\dots,x_n)^T,\ \mathbf{y} = (y_1,y_2,\dots,y_n)^T \in \mathbb{R}^n$ in non-increasing order as $x_1 \geq x_2 \geq \dots \geq x_n$ and $y_1 \geq y_2 \geq \dots \geq y_n$.

%\begin{definition}
%   Let $\mathbf{x} = (x_1,x_2,\dots,x_n)^T,\ \mathbf{y} = (y_1,y_2,\dots,y_n)^T \in \mathbb{R}^n$. If
%   \begin{gather*}
%   \sum_{i=1}^{k} x_i \leq \sum_{i=1}^{k} y_i, \ k=1,2,\dots,n,
%   \end{gather*}
%   then we say that $\mathbf{x}$ is weakly majorized by $\mathbf{y}$ and denote it by $\mathbf{x} \prec_w \mathbf{y}$. If
%   \begin{gather*}
%   \mathbf{x} \prec_w \mathbf{y} \ \text{and}  \ \sum_{i=1}^{n} x_i = \sum_{i=1}^{n} y_i,
%   \end{gather*}
%   then we say that $\mathbf{x}$ is majorized by $\mathbf{y}$ and denote it by $\mathbf{x} \prec \mathbf{y}$.
%\end{definition}

\begin{lemma}[{\cite[Theorem 3.9]{zhan2013matrix}}]\label{Stochastic-majorization}
    Let $x,y \in \mathbb{R}^n_{+}$. Then $x\prec_w y$ if and only if there exists a doubly substochastic matrix $A$ such that $x =Ay$.
\end{lemma}

For $x = (x_1,x_2,\dots,x_n) \in \mathbb{R}^n$ and $1 \leq p < \infty$, define $\Vert x \Vert_p = \Big\{\sum_{i=1}^{n} \vert x_i \vert^p \Big\}^\frac{1}{p}.$
\begin{theorem}[{\cite[Theorem 2.1]{lineigenvalues}}]\label{rev_majorization_thm}
    Let $x = (x_1,x_2,\dots,x_n),\ y = (y_1,y_2,\dots,y_n) \in \mathbb{R}^n_{+}$, such that $x_i$ and $y_i$ are in non-increasing order. If $y \prec_w x$, then $\Vert y\Vert_p \leq \Vert x\Vert_p$ for any real number $p> 1$, where equality holds if and only if $x=y$.
\end{theorem}
\begin{theorem}[The Interlacing Theorem]\label{interlacing-signed}
    Suppose $A \in \mathbb{R}^{n\times n}$ is symmetric and $B \in \mathbb{R}^{m\times m}$ be a principal submatrix of $A$. Suppose $A$ has eigenvalues $\lambda_1 \leq \lambda_2 \leq \dots \leq\lambda_n$
    and $B$ has eigenvalues $\beta_1 \leq \beta_2 \leq \dots \leq \beta_m$. Then
    $$\lambda_k \leq \beta_k \leq \lambda_{n-m+k} \quad \text{for} \ k = 1,\dots,m.$$
\end{theorem}
\section{Eigenvalues and Balanced clique number} \label{sec:frustration-bounds}

In this section, we first prove an open problem posed by Wang et al. \cite{wang2021eigenvalues}, and as a corollary, we obtain some of the known classical bounds for the least eigenvalues of the adjacency matrices of unsigned graphs. Then we prove Tur\'{a}n's theorem for signed graphs. Thereupon we prove a signed version one of Nikiforov's bound for the largest eigenvalue of unsigned graph. As a consequence of this we give an different proof for a result of Stani\'{c}. On a final note we derive an alternate proof of Motzkin-Strauss theorem for signed graphs.

For a signed graph $\Sigma$, the frustration index, denoted by  $\epsilon(\Sigma)$, is the minimum number of edges to be deleted such that the resultant signed graph is balanced \cite{belardo-laa}.

Let $\Sigma = (G , \sigma)$ be a signed graph. An edge $e$ of $\Sigma$ is \textit{positive} (resp. \textit{negative}) if $\sigma(e) =1$ (resp. $\sigma(e)=-1$). For a signed graph $\Sigma$, let $m^+(\Sigma)$  denote the number of positive edges in $\Sigma$, and $m^-(\Sigma)$ denote the number of negative edges in $\Sigma$.
\begin{lemma}\label{frustation-max-positive}
    Let $\Sigma$ be a signed graph with $n$ vertices, $m$ edges and frustration index $\epsilon(\Sigma)$. If  $\Sigma$ and $\Sigma'$ are switching equivalent, then
    $ m^{+}(\Sigma') \leq m-\epsilon(\Sigma).$
\end{lemma}
\begin{proof}
    Let $H$ be the signed subgraph of $\Sigma'$ consisting of all the positive edges of $\Sigma'$. Then $H$ is balanced, and hence $\epsilon(\Sigma')  \leq m^-(\Sigma')$. Thus $m^+(\Sigma') \leq m-\epsilon(\Sigma),$ as $\epsilon(\Sigma) = \epsilon(\Sigma')$.
\end{proof}

Essentially  Lemma \ref{frustation-max-positive} says that the  number of positive edges in any signed graph  in the whole switching equivalence class of $\Sigma = (G, \sigma)$ is bounded above by $m-\epsilon(\Sigma)$.

Next, we prove a lemma which is crucial in the proofs of the main theorems.
\begin{lemma}
    Let $\Sigma = (G, \sigma)$ be a signed graph with the balanced clique number $\omega_{b}(\Sigma)$.  If  $\Sigma$ and $\Sigma'$ are switching equivalent, then $\omega_{b}(\Sigma) = \omega_{b}(\Sigma'). $
\end{lemma}
\begin{proof}
    Since switching in signed graphs preserves the balance of cycles, the balance of cliques must also be preserved.
\end{proof}

%The follwoing theorem was conjectured by Edwards and Elphick in 1983 \cite{edwars-elphick-1983},and proved by Nikiforov  \cite{nikiforov2002some} in 2002.
%\begin{theorem}[{\cite[Theorem 2.1]{nikiforov2002some}}]
%   Let $G$  be a graph with $n$ vertices and $m$ edges. Let $\mathbf{A}(G)$ be its adjacency matrix and $\omega(G)$ be its clique number. If  $\lambda_1(G)$ is the largest eigenvalue of the adjacency matrix of $G$,
%   then $$ \lambda^2_1(G) \leq 2m\frac{\omega(G)-1}{\omega(G)}. $$
%\end{theorem}

In \cite{wang2021eigenvalues}, Wang et al. proved the Motzkin-Strauss theorem and Wilf's theorem for signed graphs. In the same paper extending Theorem \ref{nikiforov_2m} for the signed graphs is mentioned as an open problem. Next, we prove a stronger version of this open problem.
\begin{theorem}\label{nikiforov-signed-general}
    Let $\Sigma = (G, \sigma)$ be a signed graph with  $m$ edges. If $\lambda_1(\Sigma)$ is the largest eigenvalue of the adjacency matrix $A(\Sigma)$,  then $$ \lambda_1^2(\Sigma) \leq 2\big(m-\epsilon(\Sigma)\big)\bigg(1- \frac{1}{\omega_{b}(\Sigma)}\bigg), $$
    where  $\omega_{b}(\Sigma)$ and  $\epsilon(\Sigma)$ are the balanced clique number  and frustration index of $\Sigma$, respectively.
\end{theorem}

\begin{proof}
    Let $x=(x_1,\dots, x_n) $ be an eigenvector corresponding to the eigenvalue $\lambda_1(\Sigma)$ with $\Vert x \Vert_2 = 1$.  Define $y = (\vert x_1 \vert, \dots, \vert x_n \vert)$. Consider the switching function $\eta : V(G) \rightarrow \{ 1, -1\}$ defined as follows:  \begin{gather*}
        \eta (v_i) = \begin{cases}
            \sgn{x_i}\quad \text{if} \quad x_i \neq 0, \\
            1   \quad \text{otherwise}.
        \end{cases}
    \end{gather*}Now, switch the signed graph $\Sigma = (G, \sigma)$ to $\Sigma' = (G, \sigma')$  using the switching function $\eta$ as defined above.  Then $A(\Sigma') = D_{\eta}^{-1}A(\Sigma)D_{\eta}$. Note that  $\lambda_1(\Sigma) = \lambda_1(\Sigma')$ and $\omega_{b}(\Sigma)=\omega_{b}(\Sigma')$. Now,
    \begin{align*}
        \lambda_1(\Sigma) &= x^TA(\Sigma)x  \\
        &= y^TD_{\eta}^{-1}A(\Sigma)D_{\eta}y\\
        &= y^TA(\Sigma')y  \\
        &= 2\sum_{i \sim j}\sigma'_{ij}y_iy_j  ~~~~~~~~~~[\mbox{$\sigma'_{ij}$ is the $(i, j)^{th}$ entry of $A(\Sigma')$} ]\\
        &= \sum_{i \sim j,\ \sigma'_{ij}=1}2y_iy_j - \sum_{i \sim j,\ \sigma'_{ij}=-1}2y_iy_j.
    \end{align*}
    As all the entries of the vector $y$ are non-negative, we have
    \begin{gather}\label{stanic-lemma}
        \lambda_1(\Sigma) \leq \sum_{i \sim j,\ \sigma'_{ij}=1}2y_iy_j,
    \end{gather}
    and hence
    $$\lambda_1(\Sigma)^2 \leq \bigg(\sum_{i \sim j,\ \sigma'_{ij}=1}2y_iy_j\bigg)^2. $$
    Let  $\Sigma'_{+}$ be the signed subgraph of $\Sigma'$ consisting of all the positive edges of $\Sigma'$. Let $m^+$ be the number of edges in  $\Sigma'_{+}$,  and $\omega(\Sigma'_{+})$ be the  clique number of $\Sigma'_{+}$. By the Cauchy-Schwartz inequality, we have
    \begin{gather*}
        \bigg(\sum_{i \sim j,\ \sigma'_{ij}=1}2y_iy_j\bigg)^2 \leq 2m^{+}\sum_{i \sim j,\ \sigma'_{ij}=1}2y_i^2y_j^2.
    \end{gather*}
    Since all the  edges of the signed subgraph $\Sigma'_{+}$ are positive and the vector $y$ is a unit eigenvector (that is, $\Vert y \Vert_2=1$), by Theorem \ref{MS theorem},  we have
    \begin{gather*}
        \sum_{i \sim j,\ \sigma'_{ij}=1}2y_i^2y_j^2 \leq \bigg(1 - \frac{1}{\omega(\Sigma'_{+})}\bigg).
    \end{gather*}
    Thus
    $$\lambda_1(\Sigma)^2 \leq 2m^{+}\bigg(1 - \frac{1}{\omega(\Sigma'_{+})}\bigg).$$
    Since $\Sigma'_{+}$ is the subgraph of $\Sigma'$ containing all positive edges, by Lemma \ref{frustation-max-positive}, we have $m^+ \leq m-\epsilon(\Sigma)$ and $\omega(\Sigma'_{+}) \leq \omega_{b}(\Sigma)$. Thus
    $$ \lambda_1^2(\Sigma) \leq 2\big(m-\epsilon(\Sigma)\big)\bigg(1- \frac{1}{\omega_{b}(\Sigma)}\bigg). \qedhere $$
\end{proof}

Next, we prove some of the known classical bounds on the least eigenvalues of the adjacency matrices of unsigned graphs using Theorem \ref{nikiforov-signed-general}. The \textit{edge bipartiteness} of an unsigned graph $G$ is the least number of edges whose  deletion yields a bipartite graph \cite{shaun2012laa}. The following is a result of Favaron et al. \cite{favarondisc1993}, and the result was stated in a different form. We rewrite the bound in terms of edge bipartiteness.

\begin{corollary}[{\cite[Theorem 2.17]{favarondisc1993}}]\label{lambda-n-bounds}
    Let $G$ be a graph with $n$ vertices and $m$ edges. Let $A(G)$ be its  adjacency matrix and $\omega(G)$ be its clique number. Let $\epsilon_b(G)$ denote the edge bipartiteness of the graph $G$. Let $\lambda_n(G)$ be the least eigenvalue of $A(G)$. Then
    $$\lambda_n^2(G) \leq m - \epsilon_b(G).$$
\end{corollary}
\begin{proof}
    Let $\Sigma = (G,-1)$ be the signed graph with all negative edges on $G$. Then, by  Theorem \ref{nikiforov-signed-general}, we get
    $$\lambda_{1}^2(\Sigma) = \lambda_n^2(G) \leq 2(m - \epsilon(\Sigma))\bigg(1 - \frac{1}{\omega_{b}(\Sigma)}\bigg).$$
    In the signed graph $\Sigma = (G,-1)$,  triangles are not balanced, so $\omega_b(\Sigma) =2$. Note that $\Sigma$ is balanced if and only if $G$ is bipartite, thus $\epsilon(\Sigma) = \epsilon_b(G)$. Substituting these in the expression above, gives the desired result.
\end{proof}

From Corollary \ref{lambda-n-bounds}, we can also prove two other classical bounds for $\lambda_{n}(G)$ for a graph $G$.
\begin{corollary}[{\cite[Proposition 2]{gregory1985laa}}]
    Let $G$ be a graph with $n$ vertices and $m$ edges. Let $A(G)$ be its  adjacency matrix and $\omega(G)$ be its clique number. Let $ \lambda_n(G)$ be the least eigenvalue of $A(G)$.
    \begin{enumerate}
        \item If $n = 2k$, then $\vert\lambda_{n}(G)\vert \leq k.$
        \item If $n = 2k+1$, then $\vert\lambda_{n}(G)\vert \leq \sqrt{k(k+1)}.$
    \end{enumerate}
\end{corollary}
\begin{proof}
    From Corollary \ref{lambda-n-bounds}, we have
    $$\lambda_n^2(G) \leq m - \epsilon_b(G).$$
    By the definition of $\epsilon_b(G)$, $m - \epsilon_b(G)$ is the number of edges in a maximal bipartite subgraph of $G$. Let $V_1\cup V_2$ be a bipartition of the vertex set of a maximal bipartite subgraph obtained by removing $\epsilon_b(G)$ edges. Then
    $$ \vert\lambda_n(G)\vert \leq \sqrt{m- \epsilon_b(G)} \leq \sqrt{|V_1||V_2|}.$$
We have the following two cases:
    \begin{enumerate}
        \item If $|V_1| + |V_2| = n = 2k$, then by AM-GM inequality, $ \sqrt{|V_1||V_2|} \leq k.$
        \item If $|V_1| + |V_2| = n = 2k+1$, then let $|V_1| = k+a$, $|V_2| = k+1 -a$, where $a$ is an integer. Then $\sqrt{|V_1||V_2|} =  \sqrt{(k+a)(k+1-a)} = \sqrt{k(k+1)+a-a^2} \leq \sqrt{k(k+1)}$, as $a$ is an integer. \qedhere
    \end{enumerate}
\end{proof}
Next we recall the classical Tura\'{n} inequality for the unsigned graphs.

\begin{theorem}
	Let $G$ be a graph on $n$ vertices and $m$ edges with clique number $\omega$. Then $$m \leq \frac{n^2}{2}\bigg(1 - \frac{1}{\omega(\Sigma)}\bigg).$$
\end{theorem}
The following theorem extends Tur\'{a}n's inequality for the signed graphs in terms of the balanced clique number and frustration index.
\begin{theorem}
    Let $\Sigma = (G, \sigma)$ be a signed graph with $n$ vertices, $m$ edges, balanced clique number $\omega_{b}(\Sigma)$ and the frustration index $\epsilon(\Sigma)$. Then
    $$ m \leq \epsilon(\Sigma) + \frac{n^2}{2}\bigg(1 - \frac{1}{\omega_{b}(\Sigma)}\bigg). $$

\end{theorem}
\begin{proof}
    Since $\epsilon(\Sigma)$ is the frustration index, so there exists a set, say $S$,  of $\epsilon(\Sigma)$ edges such that the signed graph $\Sigma \setminus S$ is balanced. Let $\eta$ be the switching function which switches $\Sigma \setminus S$ to a signed graph with all positive edges. Let $\omega$ be the clique number of $\Sigma \setminus S$. By  Lemma \ref{frustation-max-positive} and  Tur\'{a}n's theorem for unsigned graphs, we have
    $$ m^{+}(\Sigma\setminus S) \leq m-\epsilon(\Sigma)  \leq \frac{n^2}{2}\bigg(1- \frac{1}{\omega}\bigg).$$
    As $\omega \leq \omega_{b}(\Sigma)$, we  get

    $$ m \leq \epsilon(\Sigma) + \frac{n^2}{2}\bigg(1 - \frac{1}{\omega_{b}(\Sigma)}\bigg). $$
\end{proof}

For a real number $x$, let $\floor x $ denote the greatest integer less than or equal $x$.
For a simple graph $G$ on $m$ edges, Stanley \cite{stanley1987bound} established that  $\rho(A(G))\leq \frac{1}{2}(-1+\sqrt{1+8m}).$  Later Nikiforov \cite{nikiforov2020turan} proved  the following improved bound $$\rho(A(G))\leq \sqrt{2 \Big(1-\Big\lfloor
    \frac{1}{2} + \sqrt{2m + \frac{1}{4}} \Big\rfloor ^{-1}\Big)m}.$$

We extend  Nikiforov's bound for the signed graphs in the next theorem.
The proof idea is similar to that of Nikiforov \cite{nikiforov2020turan}.
\begin{theorem}\label{Improved-stanley-bound}
    Let $\Sigma = (G, \sigma)$ be a signed graph with $n$ vertices, $m$ edges, balanced clique number $\omega_{b}(\Sigma)$ and frustration index $\epsilon(\Sigma)$. Then
    $$\lambda_1(\Sigma) \leq  \sqrt{2\big(m-\epsilon(\Sigma)\big)\bigg(1- \floor[\Big]{\frac{1}{2} + \sqrt{2(m-\epsilon(\Sigma)) + \frac{1}{4}}}^{-1}\ \bigg)},$$ where $\lambda_1(\Sigma)$ is the largest eigenvalue of $A(\Sigma).$
\end{theorem}
\begin{proof}
    Let $\Omega$ be a balanced complete subgraph of $\Sigma$ with $\omega_{b}(\Sigma)$ vertices. Let $\eta'$ be a switching function on $\Omega$ which switches all the edges of $\Omega$ to positive. Define the switching function  $\eta$ on $\Sigma$ as follows:
    \begin{gather*}
        \eta (v_i) = \begin{cases}
            \eta' (v_i)\quad \text{if} \quad v_i \in \Omega, \\
            1   \quad \text{otherwise}.
        \end{cases}
    \end{gather*}
    Let $\Sigma'$ be the signed graph obtained from $\Sigma$ by applying the switching $\eta$. Then, by Lemma \ref{frustation-max-positive},
    $$ {\omega_{b}(\Sigma)\choose 2} \leq m^{+}(\Sigma') \leq m - \epsilon(\Sigma).$$
    By solving the above expression  for $\omega_{b}(\Sigma)$, we get
    $$ \omega_b(\Sigma) \leq \frac{1}{2} + \sqrt{2(m-\epsilon(\Sigma)) + \frac{1}{4}}.$$
    The result follows by noting that $\omega_{b}(\Sigma)$ is a positive number and hence
    $$ \omega_b(\Sigma) \leq \floor[\bigg]{\frac{1}{2} + \sqrt{2(m-\epsilon(\Sigma)) + \frac{1}{4}}}. $$ Substituting this version in Theorem \ref{nikiforov-signed-general} proves the result.
\end{proof}

As a corollary of above the theorem we get the following result of Stani\'{c}.
\begin{corollary}[{\cite[Theorem 4.2]{stanic2019bounding}}]\label{eigenbound-frust}
    Let $\Sigma = (G, \sigma)$ be a signed graph with $n$ vertices, $m$ edges and frustration index $\epsilon(\Sigma)$. Then  $$ \lambda_1(\Sigma) \leq \sqrt{2(m-\epsilon(\Sigma)) +  \frac{1}{4}} - \frac{1}{2}, $$ where $\lambda_1(\Sigma) $ is the largest eigenvalue of  $A(\Sigma)$.
\end{corollary}
\begin{proof}
    From the proof of the above theorem, we have
    $$ \omega_b(\Sigma) \leq \frac{1}{2} + \sqrt{2(m-\epsilon(\Sigma)) + \frac{1}{4}}. $$
    Substituting the above bound in Theorem \ref{nikiforov-signed-general}, we get
    \begin{align*}
        \lambda_1(\Sigma) &\leq \sqrt{2\big(m-\epsilon(\Sigma)\big)\bigg(1- \frac{1}{\omega_{b}(\Sigma)}\bigg)} \\ &\leq \sqrt{2\big(m-\epsilon(\Sigma)\big)\frac{-\frac{1}{2} + \sqrt{2(m-\epsilon(\Sigma)) + \frac{1}{4}}}{\frac{1}{2} + \sqrt{2(m-\epsilon(\Sigma)) + \frac{1}{4}}}}  \\ &= -\frac{1}{2} + \sqrt{2(m-\epsilon(\Sigma)) + \frac{1}{4}}.
    \end{align*} \qedhere
\end{proof}

Let us recall the definition of the MS-index of a signed graph $\Sigma = (G, \sigma)$:
$ \mu(\Sigma) = \max_{x \in S_{n}^{\pm}} \sum_{i \sim j} \sigma(e_{ij}) x_ix_j = \max_{x \in S_{n}^{\pm}} \frac{x^TA(\Sigma)x}{2}. $
Next, we give an alternate proof for Theorem  \ref{Signed-MS theorem}.

\begin{theorem}[{\cite[Theorem 5]{wang2021eigenvalues}}]
  Let $\Sigma = (G, \sigma)$ be a signed graph with balanced clique number $\omega_{b}(\Sigma)$.
  Then $$ \mu(\Sigma) = \frac{1}{2} \bigg(\frac{\omega_b(\Sigma)-1}{\omega_b(\Sigma)}\bigg). $$
\end{theorem}

\begin{proof}
    Let $x=(x_1,\dots, x_n)  \in S_n^{\pm}$ be a vector such that  $\mu(\Sigma) =  \sum_{i \sim j} x_ix_j $.  Define $y = (\vert x_1 \vert, \dots, \vert x_n \vert)$. Consider the switching function $\eta : V(G) \rightarrow \{1, -1\}$ defined as follows:  \begin{gather*}
        \eta (v_i) = \begin{cases}
            \sgn{x_i}\quad \text{if} \quad x_i \neq 0, \\
            1   \quad \text{otherwise}.
        \end{cases}
    \end{gather*}Now, switch the signed graph $\Sigma$ to $\Sigma'$ by using the switching function $\eta$ as defined above.  Then $A(\Sigma') = D_{\eta}^{-1}A(\Sigma)D_{\eta}$. Note that  $\mu(\Sigma) = \mu(\Sigma')$ and $\omega_{b}(\Sigma)=\omega_{b}(\Sigma')$. Now,
    \begin{align*}
        2\mu(\Sigma) &= x^TA(\Sigma)x  \\
        &= y^TD_{\eta}^{-1}A(\Sigma)D_{\eta}y\\
        &= y^TA(\Sigma')y  \\
        &= 2\sum_{i \sim j}\sigma'_{ij}y_iy_j  ~~~~~~~~~~[\mbox{$\sigma'_{ij}$ is the $(i, j)^{th}$ entry of $A(\Sigma')$} ]\\
        &= \sum_{i \sim j,\ \sigma'_{ij}=1}2y_iy_j - \sum_{i \sim j,\ \sigma'_{ij}=-1}2y_iy_j.
    \end{align*}
    As all the entries of the vector $y$ are non-negative, we have
    \begin{gather*}
        2\mu(\Sigma) \leq \sum_{i \sim j,\ \sigma'_{ij}=1}2y_iy_j.
    \end{gather*}
    Let  $\Sigma'_{+}$ be the signed subgraph of $\Sigma'$ consisting of all the positive edges of $\Sigma'$. Let $m^+$ be the number of edges in  $\Sigma'_{+}$,  and $\omega(\Sigma'_{+})$ be the  clique number of $\Sigma'_{+}$. By  Theorem \ref{MS theorem}, we have
    \begin{gather*}
        \sum_{i \sim j,\ \sigma'_{ij}=1}2y_iy_j \leq 2\mu(\Sigma'_{+}) = \frac{\omega(\Sigma'_{+})-1}{\omega(\Sigma'_{+})}.
    \end{gather*}
    Since $\omega(\Sigma'_{+}) \leq \omega_{b}(\Sigma)$, we have
    $$ \mu(\Sigma) \leq \frac{1}{2}\frac{\omega_b(\Sigma)-1}{\omega_b(\Sigma)}. $$
    Now, let $\Omega$ be a balanced complete subgraph of $\Sigma$ with $\omega_{b}(\Sigma)$ vertices. Let $\eta'$ be a switching function on $\Omega$ which switches all the edges of $\Omega$ to positive. Define the switching function  $\eta$ on $\Sigma$ as follows:

    \begin{gather*}
        \eta (v_i) = \begin{cases}
            \eta' (v_i)\quad \text{if} \quad v_i \in \Omega, \\
            1   \quad \text{otherwise}.
        \end{cases}
    \end{gather*}

    Let $\Sigma''$ be the signed graph obtained from $\Sigma$ by applying the switching function $\eta$.  Let $z = (z_1,z_2,\dots,z_n)^T$ be the vector given by
    \begin{gather*}
        z_i = \begin{cases}
            \frac{1}{\omega_{b}(\Sigma)} \quad \text{if} \quad v_i \in \Omega, \\
            0   \quad \text{otherwise}.
        \end{cases}
    \end{gather*}
    Then, we have
    $$ z^TA(\Sigma')z = \frac{\omega_b(\Sigma)-1}{\omega_b(\Sigma)} \leq 2\mu(\Sigma') = 2\mu(\Sigma),$$
    and thus we have the result.
\end{proof}

\section{ Bollob\'{a}s and Nikiforov's conjecture: Signed version}\label{signed-niki}
In this section, we prove Lin et al.'s result for the signed graph with an additional restriction, and we also give a counterexample for the result when the restriction is dropped. This example also shows that the conjecture posed by Nikiforov and Bollob\'{a}s does not hold for signed graphs in general. Then we derive a bound for the largest eigenvalue of a signed graph in terms of the number of edges and the number of triangles in it.

In \cite{bollobas2007cliques}, Bollob\'{a}s and Nikiforov proposed the following conjecture.
\begin{conj}\label{bol-niki-conj}
    If $G$ is a $K_{r+1}$-free graph of order at least $r+1$ with $m$ edges, and  $\lambda_1 \geq \lambda_2 \geq \dots \geq \lambda_n$ are the eigenvalues of $A(G)$, then $\lambda_1^2 + \lambda_2^2 \leq \frac{r-1}{r}2m.$
\end{conj}
Recently Lin et al. \cite{lineigenvalues} confirmed this conjecture for the triangle-free graphs.

\begin{theorem}[{\cite[Theorem 1.2]{lineigenvalues}}]
    Let $G$ be a triangle-free graph on $n$ vertices and $m$ edges with $n \geq 3$. Then $\lambda_1^2 + \lambda_2^2 \leq m.$
\end{theorem}

The following example shows that the above theorem, as well as the conjecture of Bollob\'{a}s and Nikiforov, need not be true for signed graphs.  Let $C_5$ denote the cycle on $5$ vertices with $V(C_5) =\{v_1, v_2, v_3, v_4, v_5\}$.  Consider the signed graph $\Sigma = (C_5, \sigma)$ with  $\sigma$  assigns $-1$ to the edge between the vertices $v_1$ and $v_2$, and $1$ to the remaining edges. Then $$A(\Sigma) =
\begin{bmatrix}
    0 & -1 &  0 & 0 & 1\\
    -1 & 0 &  1 & 0 & 0\\
    0 & 1 &  0 & 1 & 0\\
    0 & 0 &  1 & 0 & 1\\
    1 & 0  &  0 & 1 & 0\\
\end{bmatrix}.
$$

The spectrum of $A(\Sigma)$ is $\{  1.618,  1.618, -0.618, -0.618, -2 \}$.  It is easy to see that $\lambda_1^2(\Sigma) + \lambda_2^2(\Sigma) =     5.2358 > 5.$ Since $r=2$ and $m =5$, so $\frac{r-1}{r} 2m = 5$. Thus Conjecture \ref{bol-niki-conj} need not be true for signed graphs.

Nevertheless, we have the following result for the signed graph.

\begin{theorem}\label{nikiforov-triangle-result}
    Let $\Sigma$ be a signed graph with $n$ vertices and $m$ edges with $n \geq 3$, and  $\Sigma$ has no  balanced triangles. If $\lambda_1(\Sigma) \geq \lambda_2(\Sigma) \geq \dots \geq \lambda_{n-1}(\Sigma) \geq \lambda_n(\Sigma)$ are the eigenvalues of  $A(\Sigma)$ and $\lambda_1(\Sigma) \geq \vert \lambda_n(\Sigma)\vert$, then $$ \lambda_1^2(\Sigma) + \lambda_2^2(\Sigma) \leq m. $$
\end{theorem}
We need the following result in the proof of the above theorem.
\begin{lemma}\label{balanced-triangle-free-lambda2}
    Let $\Sigma = (G, \sigma)$ be a signed graph on  $n$ vertices with $n \geq 3$, and  $\Sigma$ has no  balanced triangles. Let $\lambda_1(\Sigma) \geq \lambda_2(\Sigma) \geq \dots \geq \lambda_n(\Sigma)$ be the eigenvalues of $A(\Sigma)$. Then $$\lambda_2(\Sigma) \geq 0.$$
\end{lemma}
\begin{proof}
    Since $\Sigma$ does not contain balanced triangles, either $\Sigma$  is a signed graph with an induced negative triangle or its underlying graph $G$ is triangle-free.

    If $\Sigma$ has an induced negative triangle, then its spectrum is  $\{1,1,-2\}$. By Theorem \ref{interlacing-signed}, we have
    $$ \lambda_2(\Sigma) \geq 0.$$

    If $\Sigma$ is triangle free signed graph, then it has an induced path on $3$ vertices whose spectrum is $\{\sqrt{2},0,-\sqrt{2}\}$. Now, by  Theorem \ref{interlacing-signed}, we have
    $$ \lambda_2(\Sigma) \geq 0.$$
    Thus in either of the cases, we have the result.
\end{proof}

\textbf{Proof of Theorem \ref{nikiforov-triangle-result}:}
Let  $(n^{+}, n^{-}, n^{0})$ denote the inertia of $\Sigma$, where $n^{+}$, $n^{-}$ and $n^{0}$ are the numbers (counting multiplicities) of positive, negative and zero eigenvalues of $A(\Sigma)$, respectively.

Let $\lambda_1(\Sigma) \geq \vert \lambda_n(\Sigma)\vert$.  Set $S^{+} = \lambda_1^2 + \dots + \lambda^2_{n^{+}}$ and $S^{-} = \lambda_{n-n^{-}+1}^2 + \dots + \lambda^2_{n}$. Suppose that $\lambda_1^2(\Sigma) + \lambda_2^2(\Sigma) > m$. Since $S^{+} + S^{-} = 2m$, we have
\begin{gather*}
    \lambda_1^2(\Sigma) + \lambda_2^2(\Sigma) >  m = \frac{S^{+} + S^{-}}{2}
\end{gather*}
and we get,
\begin{gather*}
    \lambda_1^2(\Sigma) + \lambda_2^2(\Sigma) \geq 2(\lambda_1^2(\Sigma) + \lambda_2^2(\Sigma)) - S^{+} > S^{-} \geq 0, \\
    \lambda_1^2(\Sigma) + \lambda_2^2(\Sigma) > S^{-}.
\end{gather*}
Now we construct two $n^{-}$-vectors $x$ and $y$ such that $x = (\lambda_1^2(\Sigma), \lambda_2^2(\Sigma),0,\dots,0)^T$, $y = (\lambda_n^2(\Sigma), \lambda_{n-1}^2(\Sigma),\dots, \lambda^2_{n-n^{-}+1}(\Sigma))^T$. Since $\lambda_1^2(\Sigma) + \lambda_2^2(\Sigma) > S^{-}$, we have $y \prec_w x$ and $x \neq y$. Set $p=\frac{3}{2}$, by Theorem \ref{rev_majorization_thm}, we have
$$\Vert x \Vert^{3/2}_{3/2} > \Vert y \Vert^{3/2}_{3/2}, $$ that is,
$$ \lambda_1^3(\Sigma) +  \lambda_2^3(\Sigma) > |\lambda_n^3(\Sigma)|+ |\lambda_{n-1}^3(\Sigma)|+ \dots + |\lambda^3_{n-n^{-}+1}(\Sigma)|.$$
Since we know that in a  signed graph without balanced triangles, $$ \sum_{i=1}^n \lambda_i^3(\Sigma) = 6t_s(\Sigma) = 6(t^{+}(\Sigma) - t^{-}(\Sigma)) < 0.$$
This implies that
\begin{gather*}
    6t_s(\Sigma) = \sum_{i=1}^n \lambda_i^3(\Sigma) >  \lambda_1^3(\Sigma) + \lambda_2^3(\Sigma) + \lambda_n^3(\Sigma)+ \lambda_{n-1}^3(\Sigma)+ \dots + \lambda^3_{n-n^{-}+1}(\Sigma) > 0.
\end{gather*}
This is a contradiction. Thus we have our desired result.  \qed

Let $t_s(\Sigma) $ denote the number of positive triangles minus the number of negative triangles in $\Sigma$. The technique used in the proof of Theorem \ref{nikiforov-triangle-result} is similar to that of Lin et al. \cite{lineigenvalues}. This technique works if the assumption of balanced triangle free is replaced by the assumption that number of positive triangles is less than the number of negative triangles, that is  $t_s(\Sigma) \leq 0$.  From this we can conclude that in a given signed graph $\Sigma$, either $\lambda_{1}(\Sigma)^2 \leq m$ or $\lambda_{n}(\Sigma)^2 \leq m$, as either $t_s(\Sigma) \leq 0$ or $t_s(-\Sigma) \leq 0$ is always true.

In the following theorem, we give a different bound for the largest eigenvalue of a signed graph in terms of the total number of triangles present in the signed graph whenever the number of positive triangles is greater than the number of negative triangles in the signed graph.
\begin{theorem}\label{diff-triangle-result}
    Let $\Sigma$ be a signed graph with $n$ vertices and $m$ edges with $n \geq 3$.  Let $\lambda_1(\Sigma)$ be the largest eigenvalue of $A(\Sigma)$. If $t_s(\Sigma) \geq 0$, then $$ \lambda_1^2(\Sigma) \leq m + \Big(6t_s(\Sigma)\Big)^{\frac{2}{3}}. $$
\end{theorem}

\begin{proof}
    Let $(n^{+}, n^{-}, n^{0})$ denote the inertia of $\Sigma$, where $n^{+}$, $n^{-}$ and $n^{0}$ are the numbers (counting multiplicities) of positive, negative and zero eigenvalues of $A(\Sigma)$, respectively.

    Let $S^{+} = \lambda_1^2 + \dots + \lambda^2_{n^{+}}$ and $S^{-} = \lambda_{n-n^{-}+1}^2 + \dots + \lambda^2_{n}$. Suppose that $\lambda_1^2(\Sigma) > m +\Big(6t_s(\Sigma)\Big)^{\frac{2}{3}}$. Since $S^{+} + S^{-} = 2m$, we have
    \begin{gather*}
        \lambda_1^2(\Sigma)  >  m + \Big(6t_s(\Sigma)\Big)^{\frac{2}{3}} = \frac{S^{+} + S^{-}}{2} + \Big(6t_s(\Sigma)\Big)^{\frac{2}{3}},
    \end{gather*}
   which implies that $2\bigg(\lambda_1^2(\Sigma) - \Big(6t_s(\Sigma)\Big)^{\frac{2}{3}}\bigg) - S^{+} > S^{-} .$
   Now, it is easy to see that
         $ \lambda_1^2(\Sigma) - \Big(6t_s(\Sigma)\Big)^{\frac{2}{3}} \geq 2\bigg(\lambda_1^2(\Sigma) - \Big(6t_s(\Sigma)\Big)^{\frac{2}{3}}\bigg) - S^{+}$.   Thus \begin{gather*}
        \lambda_1^2(\Sigma) - \Big(6t_s(\Sigma)\Big)^{\frac{2}{3}} > S^{-}.
    \end{gather*}
    Define two $(n^{-}+ 1)$-vectors $x$ and $y$ as follows:  $$x = (\lambda_1^2(\Sigma), 0,0,\dots,0)^T, ~~y = (\lambda_n^2(\Sigma), \lambda_{n-1}^2(\Sigma),\dots, \lambda^2_{n-n^{-}+1}(\Sigma), \Big(6t_s(\Sigma)\Big)^{\frac{2}{3}})^T.$$ Since $  \lambda_1^2(\Sigma) - \Big(6t_s(\Sigma)\Big)^{\frac{2}{3}} > S^{-}$, we have $y \prec_w x$ and $x \neq y$. Set $p=\frac{3}{2}$, by Theorem \ref{rev_majorization_thm}, we have
    $$\Vert x \Vert^{3/2}_{3/2} > \Vert y \Vert^{3/2}_{3/2}, $$ that is,
    $$ \lambda_1^3(\Sigma) > |\lambda_n^3(\Sigma)|+ |\lambda_{n-1}^3(\Sigma)|+ \dots + |\lambda^3_{n-n^{-}+1}(\Sigma)| + 6t_s(\Sigma).$$
    Since we know that in a signed graph $ \sum_{i=1}^n \lambda_i^3(\Sigma) = 6t_s(\Sigma)$,
  we get
    \begin{gather*}
        6t_s(\Sigma) = \sum_{i=1}^n \lambda_i^3(\Sigma) >  \lambda_1^3(\Sigma) + \lambda_n^3(\Sigma)+ \lambda_{n-1}^3(\Sigma)+ \dots + \lambda^3_{n-n^{-}+1}(\Sigma)> 6t_s(\Sigma).
    \end{gather*}
    This is a contradiction. Thus we have our desired result.
\end{proof}

\begin{remark}{\rm
Note that Theorem \ref{diff-triangle-result}, in case of unsigned graphs, simplifies to the inequality
$$ \lambda_1^2(G) \leq m + \Big(6t(G)\Big)^{\frac{2}{3}}, $$
which is equivalent to Theorem \ref{nikiforov_2m} in case of triangle free graphs, but it is different when the graph is not triangle free.}
\end{remark}

\section{Signed walks and largest eigenvalue}\label{signedwalk-largesteigen}

A walk in a signed graph is positive if the number of its negative edges is even (including the repeating edges); otherwise, it is negative. In the same way we decide whether a cycle in a signed graph is positive or negative. Let $w^{+}_k(i,j)$ denote the number of positive walks of length $k-1$ starting at $i$ and terminating at $j$. Define $w^+_k(i) = w^{+}_k(i,i)$ for each $1 \leq i \leq n$.  Similarly $w^{-}_k(i,j)$ and $w^-_k(i)$ are defined in terms of the negative walks. Let $w_r^{+}(\Sigma)$ and $w_r^{-}(\Sigma)$ denote  the number of positive walks and negative walks in signed graph $\Sigma$ of length $r-1$,  respectively.

For a signed graph $\Sigma  = (G, \sigma)$,  the $(i,j)$-th entry of $A(\Sigma)^{r-1}$ is the difference between the number of positive and negative walks of length $r-1$ between the vertices $v_i$ and $v_j$ in $\Sigma$. Hence the expression $ e^TA(\Sigma)^{r-1}e$  is the difference between the number of positive and negative walks of length $r-1$ in $\Sigma$. Define $w_k(\Sigma) = e^TA(\Sigma)^{r-1}e $.

The total number of positive and negative walks in $\Sigma$ need not be switching invariant. We define \textit{$r$-frustration index }of signed graph $\Sigma$, denoted by $\epsilon_r(\Sigma)$, to be the minimum number of negative walks of length $r-1$ in the switching equivalence class of $\Sigma$. That is, $$ \epsilon_r(\Sigma) = \min\limits_{\Sigma' \sim \Sigma} \ w^{-}_r(\Sigma) .$$ Note that, $\epsilon_2(\Sigma) = \epsilon(\Sigma)$, and thus $\epsilon_r(\Sigma)$ generalizes the notion of frustration index to signed walks of higher lengths.

\begin{theorem}\label{signed-walks-gen}
    Let $\Sigma = (G,\sigma)$ be a signed graph with $n$ vertices and $m$ edges. Let $A(\Sigma)$ be its signed adjacency matrix and $\omega_b(\Sigma)$ be its balanced clique number. Let $w_r(G)$ denote the total number of walks in $G$ and $\epsilon_r(\Sigma)$ be the $r$-frustration index of \ $\Sigma$. Let $\lambda_1(\Sigma) $ be the largest eigenvalue of  $A(\Sigma)$.
    Then $$ \lambda^r_1(\Sigma) \leq \big(w_r(G) - \epsilon_r(\Sigma) \big)\bigg(1-\frac{1}{\omega_b(\Sigma)}\bigg). $$
\end{theorem}
\begin{proof}
    Let $\Sigma' = (G, \sigma')$ be the signed graph obtained $\Sigma$ by applying the switching function $\eta$ defined in Theorem  \ref{nikiforov-signed-general}.    From the equation \ref{stanic-lemma}, we have $\lambda_1(\Sigma) \leq \sum_{i \sim j,\ \sigma'_{ij}=1} 2y_iy_j.$
    Let $\Sigma'_{+}$ be the signed subgraph of $\Sigma'$ spanned by the set of edges $\sigma'_{ij} = 1$. Let the number of $r$-walks in the all positive signed subgraph $\Sigma'_{+}$ be $w_r(\Sigma'_{+})$ and its clique number be $\omega(\Sigma'_{+})$. Then $        \lambda_1(\Sigma) \leq \sum_{i \sim j,\ \sigma'_{ij}=1} 2y_iy_j \leq \lambda_1(\Sigma'_{+})$. By Theorem \ref{nikiforov-walks}, we have
    \begin{gather*}
        \lambda_1^r(\Sigma) \leq \lambda_1^r(\Sigma'_{+}) \leq w_r(\Sigma'_{+})\bigg(1 - \frac{1}{\omega(\Sigma'_{+})}\bigg).
    \end{gather*}
    Since $\Sigma'_{+}$ is all positive signed subgraph of $\Sigma'$, we have $w_r(\Sigma'_{+}) \leq w_r(G) - \epsilon_r(\Sigma)$ and $\omega_{+} \leq \omega_{b}(\Sigma)$. Thus
    $$ \lambda^r_1(\Sigma) \leq \big(w_r(G) - \epsilon_r(\Sigma)\big)\bigg(1 - \frac{1}{\omega_{b}(\Sigma)}\bigg). \ $$
\end{proof}

Note when $r = 1$, we obtain Theorem \ref{Balance-clique-n} of  Wang et al \cite{wang2021eigenvalues}.

%\textbf{I think we could remove the next corollary. }As a corollary of the above theorem, we can also get bounds for the least eigenvalues of the adjacency matrices of signed graphs.
%
%\begin{corollary}
%    Let $\Sigma = (G,\sigma)$ be a signed graph with $n$ vertices and $m$ edges. Let $A(\Sigma)$ be its signed adjacency matrix and $\omega_b(\Sigma)$ be its balanced clique number. Let $w_r(G)$ denote the total number of walks in $G$ and $\epsilon_r(\Sigma)$ be the $r$-frustration index of \ $\Sigma$. Let $\lambda_n(\Sigma)$ be the least eigenvalue of  $A(\Sigma)$.
%    Then $$ \vert\lambda^r_n(\Sigma)\vert \leq \big(w_r(G) - \epsilon_r(-\Sigma) \big)\bigg(1-\frac{1}{\omega_b(-\Sigma)}\bigg). $$
%\end{corollary}
%\begin{proof}
%Note that $\lambda_{1}(-\Sigma) = \vert\lambda_{n}(\Sigma)\vert$.  By Theorem \ref{signed-walks-gen},  we get the result.
%\end{proof}

Our next objective is to extend the left-hand side of the inequality of Theorem \ref{nikiforov-walks}. Before that, we give a theorem that will be helpful.
\begin{theorem}\label{signed-walks-lambdas}
    Let $\Sigma = (G, \sigma)$ be a signed graph of order $n$ with eigenvalues $\lambda_{1}(\Sigma) \geq \lambda_{2}(\Sigma) \geq \dots \geq \lambda_{n}(\Sigma)$ and $u_1,\dots,u_n$ be the corresponding
    orthogonal unit eigenvectors. For every $1 \leq i \leq n$, let $u_i = (u_{i1},\dots, u_{in})$ and $c_i = \big(\sum_{j=1}^nu_{ij}\big)^2$. Then
    $$ w_k(\Sigma) = \sum_{i=1}^n c_i \lambda^{k-1}_i(\Sigma).$$
\end{theorem}
\begin{proof}
    Since $A(\Sigma)$ is a symmetric matrix, by the definition of $w_k(\Sigma)$ and the spectral theorem, we have
    $$ w_k(\Sigma) = e^TA(\Sigma)^{k-1}e = (Ue)^T D(\Sigma) (Ue) .$$
    It is easy to see that, $Ue = (\sum_{j=1}^nu_{1j},\dots,\sum_{j=1}^nu_{nj})^T$, and thus the result follows.
\end{proof}

\begin{theorem}
    Let $\Sigma = (G,\sigma)$ be a signed graph with $n$ vertices and $m$ edges. Let $A(\Sigma)$ be its signed adjacency matrix. Let $w_r(\Sigma)$ denote the number of signed $r$-walks in $\Sigma$. Let $\lambda_1(\Sigma) \geq \lambda_2(\Sigma) \geq \dots \geq \lambda_n(\Sigma)$ be the eigenvalues of the adjacency matrix of $\Sigma$. Let $\rho(\Sigma)$ be the spectral radius of $A(\Sigma)$.
    Then for all $r>0$ and odd $q>0$, $$ \frac{w_{q+r}(\Sigma)}{w_{q}(\Sigma)} \leq \rho^r(\Sigma). $$
\end{theorem}
\begin{proof}
    From Theorem \ref{signed-walks-lambdas}, we have
    \begin{align*}
        \frac{w_{q+r}(\Sigma)}{w_{q}(\Sigma)} &= \frac{\sum_{i=1}^n c_i \lambda^{q+r-1}_i(\Sigma)}{\sum_{i=1}^n c_i \lambda^{q-1}_i(\Sigma)} \\ &= \rho^r(\Sigma)  \frac{\sum_{i=1}^n c_i
            \big(\frac{\lambda_i(\Sigma)}{\rho(\Sigma)}\big)^{q+r-1}}{\sum_{i=1}^n c_i \big(\frac{\lambda_i(\Sigma)}{\rho(\Sigma)}\big)^{q-1}} \\
        &\leq \rho^r(\Sigma).
    \end{align*}
\end{proof}

    % -------------References-----------------

    \bibliographystyle{amsplain}
    \bibliography{references}

\end{document}